\documentclass[12pt]{amsart}
\usepackage{amssymb,amsfonts,amsmath,amsopn,amstext,amscd,latexsym, amsthm, enumerate,mathrsfs}

\usepackage{color}
\usepackage{xcolor}
\usepackage[colorlinks=true]{hyperref}
\hypersetup{
     colorlinks   = false,
     citecolor    = blue
}

\usepackage[margin=30mm]{geometry}
\usepackage{bbm}
\usepackage[all,cmtip]{xy}

\usepackage{enumerate}
\usepackage[shortlabels]{enumitem}

\newtheorem{theorem}{Theorem}[section]
\newtheorem{thm}[theorem]{Theorem}

\newtheorem{lem}[theorem]{Lemma}

\newtheorem{prop}[theorem]{Proposition}

\newtheorem{cor}[theorem]{Corollary}

\theoremstyle{definition}

\theoremstyle{remark}

\DeclareMathOperator{\van}{\mathscr{V}}
\DeclareMathOperator{\nv}{\mathscr{N}}
\DeclareMathOperator{\pv}{{\mathscr{P}}_{\textit{v}}}
\DeclareMathOperator{\pnv}{\mathcal{P}_{\textit{n}}}

\DeclareMathOperator{\Aut}{Aut}

\DeclareMathOperator{\Irr}{Irr}

\DeclareMathOperator{\Cl}{Cl}

\DeclareMathOperator{\kernel}{Ker}

\newcommand{\la}{\lambda}

\newcommand{\Syl}{{\mathrm {Syl}}}

\DeclareMathOperator{\Sym}{\mathbf S}

\newcommand{\cd}{{\mathrm {cd}}}

\newcommand{\OO}{\mathbf{O}}
\newcommand{\Centralizer}{\mathbf{C}}
\newcommand{\Center}{\mathbf{Z}}
\newcommand{\Normalizer}{\mathbf{N}}

\def\Md#1{\text{ }(\text{\rm mod } #1)\,}

\numberwithin{equation}{section}

\newcommand{\Alt}{{\mathbf  {A}}}
\newcommand{\Out}{{\mathrm {Out}}}

\makeatletter
\@namedef{subjclassname@2020}{%
  \textup{2020} Mathematics Subject Classification}
\makeatother

\begin{document}

\title[Vanishing elements]{Proportions of vanishing elements in finite groups}

\author[L.  Morotti]{Lucia Morotti}
\address{Institut f\"{u}r Algebra, Zahlentheorie und Diskrete Mathematik, Leibniz Universit\"{a}t Hannover, 30167 Hannover, Germany}
\email{morotti@math.uni-hannover.de}

\author[H. P. Tong-Viet]{Hung P. Tong-Viet}
\address{Department of Mathematical Sciences, Binghamton University, Binghamton, NY 13902-6000, USA}
\email{tongviet@math.binghamton.edu}

\thanks{The first author was supported by the DFG grant MO 3377/1-2.}


\begin{abstract}
    In this paper, we study the proportion of vanishing elements of finite groups. We show that the proportion of vanishing elements of every finite non-abelian group is bounded below by $1/2$ and classify  all finite groups whose proportions of vanishing elements attain this bound. For symmetric groups of degree at least $5$, we show that this bound is at least $2327/2520$ which is best possible.

\end{abstract}
\maketitle
\tableofcontents


\section{Introduction}
Let $G$ be a finite group. An element $g\in G$ is called a \emph{vanishing element}  of $G$  if there exists an irreducible complex character $\chi$ of $G$ such that $\chi(g)=0$. In this case, $g$ is said to be a \emph{zero} of $\chi$. Let $\van(G)$ denote the set of all vanishing elements of $G$. Elements in the set $\nv(G)=G-\van(G)$ are called \emph{non-vanishing elements} of $G$. A classical result by Burnside states that every non-linear complex irreducible character of a finite group must vanish at some element of the group. As a consequence, if $G$ is non-abelian, then $\van(G)$ is non-empty. Zero of characters is an interesting topic in character theory. For more results concerning the influence of zeros of characters to the structure of the groups, see the recent survey paper \cite{DPS}.

In this paper, we are interested in studying the proportions of vanishing and non-vanishing elements in finite groups. To be more precise, for a finite group $G$, we  call $\pv(G)=|\van(G)|/|G|$ the proportion of vanishing elements of $G$; and  $\pnv(G)=|\nv(G)|/|G|$  the proportion of non-vanishing elements of $G$. Clearly $\pv(G)+\pnv(G)=1$ for any finite group $G$. Obviously, $\pv(G)=0$ if $G$ is abelian. If $G$ is non-abelian, then by the result of Burnside's mentioned above, $\pv(G)>0$. In general, since the identity element is a non-vanishing element of any finite group $G$, $\van(G)$ is always a proper subset of $G$ and thus $0\leq \pv(G)<1.$ 

Let $p$ be a fixed prime. Let $P_n$ denote an extra-special group of order $p^{2n+1}$, where $n\ge 1$ is an integer. Then $\van(P_n)=P_n-\Center(P_n)$ as $|\Center(P_n)|=p$ and so $\pv(P_n)=1-1/p^{2n}$. Hence $\pv(P_n)\rightarrow 1$ as $n\rightarrow \infty.$ Next, let $G$ be a fixed non-abelian group. Then $0<\pnv(G)<1$ and thus by applying Lemma \ref{lem1} repeatedly, we have  $\pnv(G^k)=\pnv(G)^k\rightarrow 0$ when $k$ approaches infinity which implies that $\pv(G^k)\rightarrow 1$ as $k\rightarrow \infty.$ On the other hand, we determine a lower bound for $\pv(G)$ when $G$ is non-abelian.

\begin{thm}\label{th:lower bound} The proportion of vanishing elements in a finite non-abelian group is at least $1/2.$
\end{thm}
  
This bound is best possible as $\pv(\Sym_3)=1/2.$ 
As a consequence of Theorem \ref{th:lower bound}, we obtain the following characterization of finite abelian groups.

\begin{cor}
Let $G$ be a finite group. Then $G$ is abelian if and only if $\pv(G)<1/2.$
\end{cor}

For a finite group $G$, we denote by  $\Irr(G)$  the set of complex irreducible characters of a finite group $G$ and $\cd(G)=\{\chi(1):\chi\in\Irr(G)\}$  the set of character degrees of $G$.  Let $k(G)=|\Irr(G)|$ be the number of conjugacy classes of $G$ and $\Cl(G)$ be the set of conjugacy classes of $G$. Let $P(G)=|\mathcal{A}|/k(G)^2$, where $\mathcal{A}$ is the set of pairs $(\chi,g^G)$ in $\Irr(G)\times \Cl(G)$ with $\chi(g)=0$. The fraction $P(G)$ is called the proportion of zeros in the character table of $G$.  Miller \cite{Miller} showed that the set $\{P(G): $ G$ \text{ is a finite group}\}$ is dense in the interval $[0,1]$. Thus the proportion of zeros in the character table behaves differently than the proportion of vanishing elements, since for example there is no finite group $G$ with $0<\pv(G)<1/2$.

In our next theorem, we completely describe the structure of finite groups $G$ with $\pv(G)=1/2.$ 

\begin{thm}\label{th: classification 1/2}
Let $G$ be a finite group. Then $\pv(G)=1/2$ if and only if $G/\Center(G)$ is a Frobenius group with kernel $F/\Center(G)$, where $F$ is abelian and $|G:F|=2$.
\end{thm}
Consequently, if $G$ is a finite group with $\pv(G)=1/2$, then $G$ is metabelian and $\cd(G)=\{1,2\}.$ 

Recall that a finite group $G$ is an \emph{almost simple group} with a non-abelian simple socle $S$ if $S\unlhd G\leq \Aut(S)$. In the next two theorems, we determine the lower bounds for the proportions of vanishing elements for symmetric groups and some other almost simple groups (and some closely related groups). Note that these bounds are best possible.

\begin{thm}\label{th: symmetric groups} Let $n\ge 1$ be an integer. Then

\begin{enumerate}[$(i)$]
\item $\pv(\Sym_n)=0$, if $n=1,2$;
\item  $\pv(\Sym_n)=1/2$, if $n=3$;
\item  $\pv(\Sym_n)=5/6$, if $n=4$;
\item  $\pv(\Sym_n)\geq \pv(\Sym_7)=2327/2520$, if $n\ge 5$;
\end{enumerate}
\end{thm}
As a consequence of Theorem \ref{th: symmetric groups} and the classification of finite simple groups, we obtain the following.

\begin{thm}\label{th: simple groups} \label{prop:almost simple with abelian quotient}
Let $G$ be a finite group and let $N\unlhd G$. If $G/N$ is almost simple with socle $S/N$ and $G/S$ is abelian, then $\pv(G)\ge \pv(\Alt_7)$.
\end{thm}
We should point out  that all finite quasi-simple groups (finite perfect groups $L$ with $L/\Center(L)$ being a non-abelian simple groups) are included in Theorem \ref{th: simple groups}.

In \cite{DPS}, it is conjectured that if $G$ is a finite non-solvable group, then $\pv(G)\ge \pv(\Alt_7)=1067/1260\approx 85\%$.  
Theorem \ref{th: simple groups} above verifies this conjecture in some cases. To provide further evidence, we prove the following result.

\begin{thm}\label{th:small}
Let $G$ be a finite group. If $\pv(G)\leq 2/3$, then $G$ is  solvable. Moreover, if $\pv(G)<2/3$, then $G$ is abelian or $\pv(G)=1/2.$
\end{thm}

Calculation with GAP \cite{GAP} seems to suggest that $$\{\pv(G): G \text{ is a finite group}\}\cap [0,1067/1260)= \{(m-1)/m: 1\leq m\leq 6\}.$$
It is easy to find a finite group whose proportion of vanishing elements is one of the rational values in the above set. (See Lemmas \ref{lem2} and \ref{lem3}).

The paper is organized as follows. Theorems \ref{th: symmetric groups}  and \ref{th: simple groups} are proved in Section \ref{sec2} and the remaining theorems in Section \ref{sec3}, after some preliminary results are presented in Section \ref{sec:prelim}.

Our notation is standard. We follow \cite{Isaacs} for the character theory of finite groups and \cite{JK} for the representation theory of symmetric groups. 
\section{Preliminaries}\label{sec:prelim}
In this section, we collect and prove some results which will be needed in our proofs of the main theorems.

\begin{lem}\label{lem1}
Let $G=A\times B$ be a direct product of two finite groups $A$ and $B$. Then $\pnv(G)=\pnv(A)\pnv(B)$ and $\pv(G)=\pv(A)+\pv(B)-\pv(A)\pv(B)$. 
\end{lem}

\begin{proof}
Since $\Irr(G)=\{\theta\times \lambda:\theta\in \Irr(A), \lambda\in\Irr(B)\}$, we see that $\nv(G)=\nv(A)\times \nv(B)$ and the results follow.
\end{proof}

For a finite group $G$ and a prime $p$, an element $x\in G$ is said to be $p$-singular if $p$ divides the order of $g$. An irreducible character $\chi $ of $G$ is said to have $p$-defect zero if $p$ does not divide $|G|/\chi(1)$. By a result of R. Brauer (\cite[Theorem 8.17]{Isaacs}), if $\chi\in\Irr(G)$ has $p$-defect zero, then $\chi(g)=0$ for every $p$-singular element $g\in G$.
The next lemma is well-known, for the reader's convenience, we include its proof here.

\begin{lem}\label{lem:defect 0}
Let $G$ be a finite group and let $N\unlhd G$. Let $p$ be a prime dividing $|N|$ and let $\theta\in
\Irr(N)$ be  of $p$-defect zero. If $\chi\in\Irr(G)$ lies above $\theta$, then $\chi$ vanishes on all $p$-singular elements  of $N.$  In particular, if $N$ is non-solvable, then $N\cap \van(G)\neq\emptyset.$
\end{lem}

\begin{proof} Note that if $\theta\in\Irr(N)$ has $p$-defect zero, then every $G$-conjugate of $\theta$ also has $p$-defect zero. Now let $\chi\in\Irr(G)$ be lying above $\theta$, where $\theta\in\Irr(N)$ has $p$-defect zero. By \cite[Theorem 6.2]{Isaacs}, $\chi_N=e\sum_{i=1}^t\theta_i$, where each $\theta_i$ is a $G$-conjugate of $\theta=\theta_1$. It follows that $\chi(x)=e\sum_{i=1}^t\theta_i(x)=0$ for every $p$-singular element $x\in N.$

Now assume that $N\unlhd G$ is non-solvable. There exist normal subgroups $L\leq M\leq N$ of $G$ such that $M/L$ is a non-abelian chief factor of $G$. We have that $M/L$ is a non-abelian minimal normal subgroup of $G/L$.  If $M/L\cap \van(G/L)\neq \emptyset,$ then $M\cap \van(G)\neq\emptyset$ and we are done. So, working by induction we may assume that $L=1$.
We can write $M\cong S^n,$ where $S$ is a non-abelian simple group and $n\ge 1$ is an integer. Let $p\ge 5$ be a prime divisor of $|S|$. From \cite[Corollary 2]{GO}, $S$ has a character $\lambda\in\Irr(S)$ of $p$-defect zero. Let $\phi=\lambda^n=\lambda\times\cdots\times\lambda\in\Irr(M)$. Then $\phi$ has $p$-defect zero. Now let $\chi\in\Irr(G)$ be an irreducible constituent of  $\phi^G$. From the previous claim, we have $\chi(x)=0$ for every $p$-singular element $x\in M.$
\end{proof}

 The following is key to our proofs.

\begin{lem}\label{lem:quotient}
Let $G$ be a finite group and let $N$ be a normal subgroup of $G$. Then $$\pv(G)\ge \pv(G/N)+|N\cap\van(G)|/|G|\ge \pv(G/N).$$ Moreover, if $N$ is non-solvable, then $\pv(G)>\pv(G/N)$.
\end{lem}

\begin{proof}
Observe that if $\chi$ is an irreducible character of $G/N$, then $\chi$ can be considered as an irreducible character of $G$ with $N\subseteq \kernel(\chi)$. Moreover, if $gN\in G/N$ such that $\chi(gN)=0$, then $\chi(gn)=\chi(gN)=0$ for all $n\in N$ and thus $gN\subseteq \van(G)$.  It  follows that $$\mathcal{C}:=\bigcup_{gN\in\van(G/N)}gN\subseteq \van(G).$$ Note that if $gN\in\van(G/N)$, then $g\in G-N$. In particular, $\mathcal{C}$ and $N\cap\van(G)$ are disjoint subsets of $\van(G)$. Hence   $|\van(G)|\ge |N||\van(G/N)|+|N\cap\van(G)|$ and the first part of the lemma follows by dividing both sides by $|G|$. Now, if $N$ is non-solvable, then $N\cap \van(G)\neq \emptyset$ by Lemma \ref{lem:defect 0}; therefore, $\pv(G)>\pv(G/N)$.
\end{proof}

We will use the previous lemma together with the following result which describes the structure of  minimal non-abelian solvable groups. 
\begin{lem}\label{lem:minimal nonabelian}
Let $G$ be a finite solvable group and assume that $G'$ is the unique minimal normal subgroup of $G$. Then  all non-linear irreducible characters of $G$ have the same degree $f$ and one of the following cases holds.

\begin{enumerate}[$(a)$]
\item $G$ is a $p$-group for some prime $p$, $\Center(G)$ is cyclic, $G/\Center(G)$ is elementary abelian of order $f^2$.
\item $G$ is a Frobenius group with an abelian Frobenius complement of order $f$ and kernel $G'$ which is an elementary abelian $p$-group.
\end{enumerate}
Moreover, if $\chi$ is an irreducible character of $G$ of degree $f$, then $\chi$ vanishes on $G-\Center(G) $ in case $(a)$ and on $G-G'$ in case $(b)$.
\end{lem}

\begin{proof} This follows from \cite[Lemma 12.3]{Isaacs} and its proof.
\end{proof}

The non-vanishing elements in finite solvable groups were studied in \cite{INW}. Using their results, we can compute the proportions of vanishing elements in finite $p$-groups and some Frobenius groups. Note that if $G$ is a finite group, then $\Center(G)\subseteq\nv(G)$. The converse is also true for finite $p$-groups.

\begin{lem}\label{lem2}
Let $G$ be a finite non-abelian $p$-group. Then $\van(G)=G-\Center(G)$ and $$\pv(G)=1-\frac{1}{|G:\Center(G)|}\ge 1-\frac{1}{p^2}.$$
\end{lem}

\begin{proof} By \cite[Theorem B]{INW}, every non-vanishing element of $G$ lies in $\Center(G)$ and thus $\nv(G)=\Center(G)$ which implies that $\van(G)=G-\Center(G)$. It follows that \[\pv(G)=\frac{|\van(G)|}{|G|}=\frac{|G-\Center(G)|}{|G|}=\frac{|G|-|\Center(G)|}{|G|}=1-\frac{1}{|G:\Center(G)|}.\]
Since $G$ is non-abelian, $G/\Center(G)$ is non-cyclic, hence $|G:\Center(G)|\ge p^2$ which forces $\pv(G)=1-1/|G:\Center(G)|\ge 1-1/p^2$ as wanted.
\end{proof}

\begin{lem}\label{lem3}  Let $G$ be  a Frobenius group with kernel $F$  and let $n=|G:F|\ge 2$. Then  $G-F\subseteq \van(G)$ and $\pv(G)\geq (n-1)/n.$ Moreover, if $n=2$ or $F$ is an abelian $p$-group for some prime $p$, then $\van(G)=G-F$ and $\pv(G)= 1-1/n$.
\end{lem}

\begin{proof}
 Let $\lambda$ be a non-principal irreducible character of $F$. Then $\chi=\lambda^G$ is an irreducible character of $G$. Thus $\chi$ vanishes on $G-F$. Therefore $G-F\subseteq \van(G)$, which implies that $\pv(G)\ge (n-1)/n$ and that $\nv(G)\subseteq F$. 

Assume that $F$ is an abelian $p$-group. Then  $F$ is a normal Sylow $p$-subgroup of $G$. By \cite[Theorem A]{INW}, all elements in $F=\Center(F)$ are non-vanishing elements in $G$. Hence $G-F=\van(G)$ and $\pv(G)=|G-F|/|G|=1-1/|G:F|=1-1/n$. 

Assume next that $n=2$. Then $G=\langle h\rangle F$, where $\langle h\rangle$ is a Frobenius complement of $G$ of order $2$.  In this case, $F$ is abelian by \cite[Lemma 7.21]{Isaacs}. Hence $h$ inverts each element of $F$. Now let $x\in F$ and let $\chi$ be any non-linear irreducible character of $G$. Then $\chi=\lambda^G$ for some linear character $\lambda$ of $F$. Therefore, since $o(x)$ is odd, we have $$\chi(x)=\lambda(x)+\lambda^h(x)=\lambda(x)+\lambda(hxh^{-1})=\lambda(x)+\lambda(x^{-1})=2\textrm{Re}(\lambda(x))\neq 0.$$ Thus $\van(G)=G-F$ as $\nv(G)=F.$ The lemma now follows.
\end{proof}

%

\begin{lem}\label{lem:abelian} Let $G$ be a finite non-abelian group and let $N\unlhd G$. Suppose that $G/N$  is a Frobenius group with Frobenius kernel $F/N$. Assume further that $|G:F|=2$ and $F$ is abelian. Let $U\in\Syl_2(F)$. If $U$  is central in $G$ or if $\pv(G)<3/4$, then $\pv(G)=1/2$.
\end{lem}

\begin{proof}  Let $T$ be a $2$-complement of $F$. Then $F=T\times U.$

(i) Assume that $U\leq \Center(G)$.   Since $G/N$ is a Frobenius group with a Frobenius complement of order $2$, $\van(G/N)=G/N-F/N$ by Lemma \ref{lem3} and so $G-F\subseteq\van(G)$. To show that $\pv(G)=1/2$, we only need to show that every element of $F$ is a non-vanishing element of $G$. Fix $g\in F.$ Then $g=tu$, where $t\in T$ and $u\in U$. Hence $m:=o(t)$ is odd and $u\in U\subseteq \Center(G)$. 

Since $F\unlhd G$ is abelian of index $2$, by  It\^{o}'s theorem \cite[Theorem 6.15]{Isaacs}, every non-linear irreducible character of $G$ has degree $2$. Let $\chi\in\Irr(G)$. If $\chi(1)=1$, then $\chi(g)\neq 0$. Assume that $\chi(1)=2$. Then $\chi_F=\psi_1+\psi_2$, where $\psi_1,\psi_2\in\Irr(F)$ are $G$-conjugate linear characters of $F$ by Clifford's theorem. We also have $\chi_U=\chi(1)\lambda$ for some $\lambda\in\Irr(U)$ since $U\leq \Center(G)$. So $\chi(g)=\lambda(u)\chi(t)$. Clearly $\lambda(u)\neq 0$. Hence we need to show that $\chi(t)\neq 0.$ Now $t\in F$ and thus $\chi(t)=\psi_1(t)+\psi_2(t)$, where $\psi_i(t)'s$ are $m$th-roots of unity. However, as $m$ is odd, the sum of two $m$th-roots of unity cannot be zero. Thus $g\in F$ is a non-vanishing element of $G$. Therefore $G-F=\van(G)$ and hence $\pv(G)=1/2.$

(ii) Assume that $\pv(G)<3/4.$ Since $F/N$ is of odd order, we know that $U\leq N.$ We will show that $U\leq \Center(G)$ and hence $\pv(G)=1/2$ by part (i).   Since $T=\OO_{2'}(F)$ and $U=\OO_2(F)$,  both $T$ and $U$ are normal in $G$. Furthermore, as $|G:F|=2$, $G/T $ is a $2$-group. If $G/T$ is non-abelian, then $\pv(G/T)\ge 3/4$ by Lemma \ref{lem2}, hence $\pv(G)\ge 3/4$, contradicting our assumption. Hence $G/T$ is abelian. It follows that $[G,U]\leq T\cap U=1$ and hence $U$ is central in $G$.
\end{proof}

\section{Proportions of vanishing elements in symmetric and simple groups}\label{sec2}

The following results compares vanishing elements of symmetric and alternating groups.
\begin{lem}\label{lem: Alternating groups} Let $n\ge 3$ be an integer. Then \[\van(\Sym_n)=(\Sym_n-\Alt_n)\cup \van(\Alt_n)\text{ and } \pv(\Alt_n)=2\pv(\Sym_n)-1.\]
\end{lem}

\begin{proof} It suffices to show that $\nv(\Sym_n)=\nv(\Alt_n)$.
Since $n\ge 3$, $\Sym_n$ possesses a self-conjugate partition $\lambda$ of $n$ and so $\chi^\lambda$, when restricted to $\Alt_n$, splits into the sum of two irreducible characters, say $\chi^{\lambda\pm}$ of the same degree. Thus $\chi^\lambda=\mu^{\Sym_n}$ for each $\mu\in\{\chi^{\lambda\pm}\}$. It follows that $\chi^\lambda$ vanishes on $\Sym_n-\Alt_n$ and so $\Sym_n-\Alt_n\subseteq \van(\Sym_n)$. Therefore, $\nv(\Sym_n)\subseteq \Alt_n$.

Fix now $g\in\Alt_n$ and a partition $\la$ of $n$. Let $\theta$ be any irreducible character appearing in the restriction of $\chi^\la$ to $\Alt_n$. We will prove that $\theta(g)=0$ if and only if $\chi^\la(g)=0$. Together with the previous paragraph this will prove that $\nv(\Sym_n)=\nv(\Alt_n)$. If $\la\not=\la'$ then $\theta=(\chi^\lambda)_{\Alt_n}$, so clearly $\theta(g)=0$ if and only if $\chi^\la(g)=0$. Thus we may assume that $\la=\la'$, in which case $\theta=\chi^{\la\pm}$. Let $h(\lambda)=(h_{11},h_{22},\cdots,h_{kk})$ be the partition of $n$, where $h_{ij}$ is the $(i,j)$-hook length at the position $(i,j)$ of the Young diagram $[\lambda]$ of $\lambda$ and $k$  is the length of the main diagonal of $[\lambda]$. By \cite[Corollary 2.4.8, Theorem 2.5.13]{JK} if the cycle partition of $g$  is different from $h(\lambda)$, then $\theta(g)=\chi^\lambda(g)/2$, so in this case $\theta(g)=0$ if and only if $\chi^\la(g)=0$. If on the other hand $g$ has cycle partition $h(\la)$, then 
 $\chi^\la(g)=t$ and $\theta(g)=(t\pm\sqrt{t\prod_{i=1}^kh_{ii}})/2$ for some $t\in\{\pm 1\}$. Since $n\geq 3$, so that $\prod_{i=1}^kh_{ii}>1$, and then in this case $\theta(g)\not=0$ and $\chi^\la(g)\not=0$. So $\nv(\Sym_n)=\nv(\Alt_n)$.

Since $\Sym_n$ is a disjoint union of $\nv(\Sym_n)$ and $\van(\Sym_n)$, the lemma follows.
\end{proof}

We are now ready to prove Theorem \ref{th: symmetric groups}.

\begin{proof}[\textbf{Proof of Theorem \ref{th: symmetric groups}}] The result is clear if $n\le 4$ or $n=7$, so we may assume that $n\geq 5$ with $n\not=7$. Since $\pv(\Sym_n)+\pnv(\Sym_n)=1$, it suffices to show that $\pnv(\Sym_n)\leq \pnv(\Sym_7)=193/2520$.

Assume that for some $a\geq 0$ and $r\geq 1$ there exists an $r$-core partition $\mu=(\mu_1,\ldots,\mu_h)$ of $n-ar$ (that is a partition with no hook length divisible by $r$). Let $\la:=(\mu_1+ar,\mu_2,\ldots,\mu_h)$. If $g\in\Sym_n$ has more than $a$ cycles of length $r$, then $\chi^\la(g)=0$ by \cite[2.4.7]{JK}. 

By \cite[Theorem 1]{GO} we then have that if $g\in\nv(\Sym_n)$, then $g$ has cycle partition $(3^a,2^b,1^{n-3a-2b})$ for some $a,b\geq 0$. We will first find bounds on $a$ and $b$ by studying $2$- and $3$-cores partitions.

Let $\tau(k):=(k,k-1,\ldots,1)$. Then $\tau(k)$ is a $2$-core and $|\tau(k)|=k(k+1)/2$ (where for any partition $\psi$, $|\psi|$ is the sum of the parts of $\psi$). So $|\tau(k)|$ is even if $k\equiv 0,1\Md{4}$ and $|\tau(k)|$ is odd else. Let $k$ be maximal such that $|\tau(k)|\equiv n\Md{2}$ and $|\tau(k)|\leq n$. Then $n\leq |\tau(k+3)|-2$ (since for some $1\leq x\leq 3$ we have that $|\tau(k+x)|\equiv n\Md{2}$). Further
\[\frac{|\tau(k+3)|-2-|\tau(k)|}{2}=\frac{3k+4}{2}\leq\frac{3}{\sqrt{2}}\sqrt{|\tau(k+3)|-2}.\]
Since $|\tau(k)|\leq n\leq |\tau(k+3)|-2$, it follows that
\[\frac{n-|\tau(k)|}{2}\leq\frac{3}{\sqrt{2}}\sqrt{n}.\]

Consider now $3$-core partitions. For $k\geq 0$ let $\varphi(k):=(2k,2k-2,\ldots,2)$ and for $k\geq 1$ let $\psi(k):=(2k-1,2k-3,\ldots,1)$. Then $|\varphi(k)|=k(k+1)$, so $|\varphi(k)|\equiv 2\Md{3}$ if $k\equiv 1\Md{3}$ and $|\varphi(k)|\equiv 0\Md{3}$ else, while $|\psi(k)|=k^2$, so $|\psi(k)|\equiv 0\Md{3}$ if $k\equiv 0\Md{3}$ and $|\psi(k)|\equiv 1\Md{3}$ else. Further
\begin{align*}
\frac{|\varphi(k+3)|-3-|\varphi(k)|}{3}&=2k+3\leq2\sqrt{|\varphi(k+3)|-3},\\
\frac{|\psi(k+3)|-3-|\psi(k)|}{3}&=2k+2\leq2\sqrt{|\psi(k+3)|-3}.
\end{align*}
In particular there exists a $3$-core $\xi$ with $|\xi|=n-3m$ with $0\leq m\leq 2\sqrt{n}$.

If $g\in\nv(\Sym_n)$ has cycle partition $(3^a,2^b,1^{n-3a-2b})$, it then follows that $a\leq\lfloor 2\sqrt{n}\rfloor$ and $b\leq \lfloor 3\sqrt{n/2}\rfloor$. In particular $g$ moves at most $m_n$ points, where
\[m_n:=3\lfloor 2\sqrt{n}\rfloor+2\lfloor 3\sqrt{n/2}\rfloor.\]
So $g$ is contained in one of the $\binom{n}{m_n}$ $\Sym_n$-conjugates of $\Sym_{m_n}$. If $n\geq 106$ then $n-m_n\geq 4$, so
\[\pnv(\Sym_n)\leq\frac{\binom{n}{m_n}|\Sym_{m_n}|}{|\Sym_n|}=\frac{1}{(n-m_n)!}\leq \frac{1}{4!}<193/2520.\]

So we only still have to consider $5\leq n\leq 105$ with $n\not=7$. In this case better upper bounds for $a$ and $b$ can be found by studying small $2$- and $3$-cores (for $3$-cores note that any partition of the form $(2c+d,2c+d-2,\ldots,d+2,d^e,(d-1)^2,(d-2)^2,\ldots,1^2)$ with $c,d\geq 0$ and $1\leq e\leq 2$ is a $3$-core). Further $b$ is even, since by Lemma \ref{lem: Alternating groups} non-vanishing elements of $\Sym_n$ are contained in $\Alt_n$. This allows to check that $n-\overline{m}_n\geq 4$, where $\overline{m}_n$ is the maximal number of moved points by any non-vanishing element of $\Sym_n$, unless possibly if $n\in\{11,13\}$. Reasoning as above we may then assume that $n\in\{11,13\}$, in which case $a\in\{0,1\}$ and $b\in\{0,2,4\}$ and then the theorem holds by considering the sizes of the corresponding conjugacy classes.
\end{proof}

Note that the proof of this theorem actually proves that $\pnv(\Sym_n)\to 0$, that is $\pv(\Sym_n)\to 1$. Further the number of non-vanishing conjugacy classes of $\Sym_n$ can be bounded above by $Cn$ for some constant $C$.

\begin{proof}[\textbf{ Proof of Theorem \ref{th: simple groups}}]
In view of Lemma \ref{lem:quotient}, we may assume that $N=1$, so $G$ is almost simple with socle $S$ and $G/S$ is abelian.

It is well-known that (see \cite[Corollary 2]{GO}) for 
every prime divisor $p$ of $|S|$, $S$ has an irreducible character of $p$-defect zero except for the following cases:
\begin{enumerate}
\item $p=2$ and $S$ is isomorphic to one of the following simple groups: $\textrm{M}_{12}, \textrm{M}_{22}, \textrm{M}_{24},$ $ \textrm{J}_{2}, \textrm{HS}, \textrm{Suz}, \textrm{Ru}, \textrm{Co}_{1},\textrm{Co}_{3}, \textrm{B}$ and $\Alt_n$ for some integer $n\ge 5.$
\item $p=3$ and $S$ is isomorphic to $\textrm{Suz},\textrm{Co}_{3}$ or $\Alt_n$ for certain integer $n\ge 5.$
\end{enumerate} 

It follows from \cite[Theorem 8.17]{Isaacs} that if $S$ has an irreducible character of $p$-defect zero for each prime divisor of $|S|$, then $\van(S)=S-\{1\}$ and thus $$\pv(S)=(|S|-1)/|S|=1-1/|S|\ge 1-1/60=59/60>1067/1260.$$

If $S$ is one of the sporadic simple groups in Cases (i) and (ii) above, then we can check directly using \cite{GAP} that $\pv(S)>1067/1260$. Assume now that $S\cong\Alt_n$ with $n\ge 5.$  By Lemma \ref{lem: Alternating groups} and Theorem \ref{th: symmetric groups}, we have  $$\pv(\Alt_n)=2\pv(\Sym_n)-1\ge  2\cdot\frac{2327}{2520}-1=\frac{1067}{1260}=\pv(\Alt_7).$$ 

Thus we may now assume that $|G:S|>1$. Let $g\in G-S$.  Then $gS\in G/S$ induces an outer automorphism, say $\alpha$, of $S$. By \cite[Theorem C]{FS}, $\alpha$ does not fix some conjugacy class of $S$. By Brauer's permutation lemma \cite[Theorem 6.32]{Isaacs}, $\alpha$ does not fix some irreducible character, say $\theta$, of $S$. Now let $I=I_G(\theta)$ be the inertial group of $\theta$ in $G$. Since $S=G'\leq I\neq G$, $I$ is a proper normal subgroup of $G$. Therefore $|G:I|\ge 2.$ Let $\varphi\in\Irr(I)$ be lying above $\theta$. Then $\chi=\varphi^G\in\Irr(G)$ by Clifford's corresponding theorem \cite[Theorem 6.11]{Isaacs}. It follows that $\chi$ vanishes on the set $G-I$. Thus $\chi(g)=0$. Hence $G-S\subseteq\van(G)$.

If $S$ has a defect zero irreducible character for each prime $p$, then $S-\{1\}\subseteq \van(G)$ and so $\van(G)=G-\{1\}$ by Lemma \ref{lem:defect 0}. If $S$ is one of the sporadic simple groups in cases (i) and (ii) the theorem can be checked using \cite{GAP}. So we may assume that $S\cong\Alt_n$ with $n\ge 5$. Since $\Alt_6$ has defect zero characters for any $p$, we may also assume that $n\not=6$, so $G\cong\Sym_n$ in which case the result follows by Theorem \ref{th: symmetric groups}.
\end{proof}


\section{Lower bound for proportions of vanishing elements in finite non-abelian groups}\label{sec3}

We now prove the remaining results in the introduction.

Let $G$ be a  finite non-abelian group. Let $N\unlhd G$ be maximal such that $G/N$ is non-abelian.  Then $(G/N)'$ is a unique minimal normal subgroup of $G$. If $G/N$ is solvable, then we can apply Lemma \ref{lem:minimal nonabelian} together with Lemmas \ref{lem2} and \ref{lem3} to obtain a lower bound for $\pv(G)$. We next consider the case when $G/N$ is non-solvable. In this situation, $(G/N)'$ is perfect.

\begin{prop}\label{prop:perfect derived subgroup}
Let $G$ be a finite non-abelian group. Suppose that $G'=G''$ is perfect. Then $\pv(G)>3/4.$
\end{prop}

\begin{proof} We proceed by induction on $|G|$. Since $G$ is non-abelian and $G'$ is perfect, $G'$ is non-abelian. Let $G'/N$ be a chief factor of $G$. Since $G'$ is perfect, $G'/N$ is perfect and so $G/N$ is non-solvable. If $N$ is non-trivial, then by Lemma \ref{lem:quotient} and induction, we have $\pv(G)\ge \pv(G/N)>3/4.$ Thus we may assume that $N=1$ and so $G'$ is a minimal normal subgroup of $G$.  Let $C=\Centralizer_G(G').$ Then $G/C$ also satisfies the hypothesis of the proposition and if $C\neq 1$, then $\pv(G)\ge \pv(G/C)>3/4$ again. Therefore, we can assume that $C=1$ and hence $G'$ is the unique minimal normal subgroup of $G$. Write $G'=S_1\times S_2\times\dots\times S_n$, where $S_i\cong S$ for $1\leq i\leq n$, $S$ is a non-abelian simple group and $n\ge 1$ is an integer.

If $n=1$, then $G$ is almost simple with simple socle $G'$ and the result follows by applying Theorem \ref{prop:almost simple with abelian quotient}. So we may assume that $n\ge 2.$ Let $\Omega=\{S_1,S_2,\dots,S_n\}$. Then $G$ acts transitively on $\Omega$ with a point stabilizer $B:=\Normalizer_G(S_1).$ Note that $G'\leq B\unlhd G$ and thus $B=\Normalizer_G(S_i) $ for all $i$. Hence $B$ is also the kernel of the action of $G$ on $\Omega$ and $|G:B|=n.$  

 Let $\lambda\in\Irr(S_1)$ be a non-principal character and let $\theta=\lambda\times 1\times\cdots\times 1\in\Irr(G')$. We see that $I_G(\theta)\leq B$. Let $\varphi\in\Irr(B)$ be lying above $\theta$. It follows from Clifford's theory that $\varphi^G\in\Irr(G)$. Since $B\unlhd G$, $\varphi^G$ vanishes on $G-B$ and so $G-B\subseteq \van(G)$. 
 
Let $g\in B-G'.$ Since $G'\leq S_i\Centralizer_G(S_i)\unlhd G$ and $G$ acts transitively on $\Omega$, we deduce that $G'=S_1\Centralizer_G(S_1)$. Since $g\in B-G'=\Normalizer_G(S_1)-S_1\Centralizer_G(S_1)$,   $g$  induces a nontrivial outer automorphism on $S_1$.  By Brauer's permutation lemma and \cite[Theorem C]{FS}, there exists $\mu\in\Irr(S_1) $ such that $\mu^g\neq \mu.$ Let $\phi=\mu\times 1\times \dots\times 1\in\Irr(G')$. We see that $\phi^g\neq \phi$ and so $g\in G-I_G(\phi)$, where $G'\leq I_G(\phi)\unlhd G$. Again, if $\psi\in\Irr(I_G(\phi))$ lying above $\phi$, then $\psi^G\in\Irr(G)$ and thus $\psi^G$ vanishes on $G-I_G(\phi)$. In particular, $g\in\van(G)$ and thus $B-G'\subseteq \van(G).$

We have shown that $G-G'\subseteq \van(G)$. Since $G'$ is non-solvable, $G'\cap \van(G)\neq \emptyset$ by Lemma \ref{lem:defect 0}. It follows that $G-G'$ is a proper subset of $\van(G)$ and hence $$\pv(G)> |G-G'|/|G|=(k-1)/k,$$ where $k=|G:G'|=|G:B|\cdot |B:G'|=n|B:G'|$. If $k\ge 4$, then $(k-1)/k\ge3/4$ and so $\pv(G)>3/4$ as wanted. Thus we can assume that $k\leq 3.$ As $k=n|B:G'|$ and $n\ge 2,$ we must have that $k=n\in \{2,3\}$ and $B=G'.$ 

We next claim that $(G-G')\cup \van(G')\subseteq \van(G)$. Fix $g\in \van(G')$ and $\theta=\theta_1\times\theta_2\times\cdots\times\theta_n\in\Irr(G')$ with $\theta(g)=\prod_{i=1}^n\theta_i(g_i)=0$, where $g=(g_1, g_2,\dots, g_n)\in S_1\times S_2\times\dots\times S_n$ and $\theta_i\in \Irr(S_i)$, $1\leq i\leq n.$ It follows that $\theta_j(g_j)=0$ for some $j.$ Now let $\phi=\theta_j\times\theta_j\times\cdots\times\theta_j\in\Irr(G')$. Then $\phi$ is $G$-invariant and since $G/G'$ is cyclic, $\phi$ extends to $\phi_0\in\Irr(G)$ by \cite[Corollary 11.22]{Isaacs}. Clearly $\phi_0(g)=\phi(g)=0$ and so $g\in\van(G)$. Thus $$\pv(G)\ge (|G-G'|+|\van(G')|)/|G|=(n-1)/n+\pv(G')/n.$$ Note that $\pv(G')\ge \pv(S)>1/2$ by Lemma \ref{lem:quotient} and Theorem \ref{th: simple groups}. Hence, as $2n\ge 4,$ $$\pv(G)>(n-1)/n+1/2n=(2n-1)/2n\ge 3/4.$$ The proof is now complete.
 \end{proof}

From the proofs of Theorem \ref{prop:almost simple with abelian quotient} and Proposition \ref{prop:perfect derived subgroup} for non-solvable groups and  Lemma  \ref{lem:minimal nonabelian} for solvable groups, we obtain the following.
\begin{cor}\label{cor1} Let $G$ be a finite group. Suppose that $G'$ is the unique minimal normal subgroup of $G$ and that $|G:G'|>1$. Then there exist a character $\psi\in\Irr(G)$ and  a subgroup $G'\leq L\leq G$ with $|G:L|>1$ such that $\psi$ vanishes on $G-L.$
\end{cor}

\begin{proof}[\textbf{Proof of Theorem \ref{th:lower bound}}] We will prove the theorem by induction on $|G|$. Let $G$ be a finite non-abelian group.  Let $N$ be a minimal normal subgroup of $G$. If $G/N$ is non-abelian, then $\pv(G/N)\ge 1/2$ by induction since $|G/N|<|G|$ and thus $\pv(G)\ge \pv(G/N)\ge 1/2$ by Lemma \ref{lem:quotient}. Therefore $G/N$ is abelian for every  nontrivial normal subgroup $N$ of $G$. It follows that $G'$ is the unique minimal normal subgroup of $G$.

Assume  that $G$ is non-solvable. Then $G'$ is perfect and so by Proposition \ref{prop:perfect derived subgroup}, we have $\pv(G)>3/4>1/2.$ Assume next that $G$ is solvable.  By Lemma \ref{lem:minimal nonabelian}, $G$ is either a non-abelian $p$-group for some prime $p$ or $G$ is a Frobenius group with  Frobenius kernel  $G'$ and  $n=|G:G'|\ge 2$. If the former case holds, then $\pv(G)\ge 1-1/p^2\ge 1-1/4=3/4>1/2$ by Lemma \ref{lem2}. In the latter case, we have $\pv(G)=1-1/n\ge 1/2$ by Lemma \ref{lem3}. 
\end{proof}

\begin{lem}\label{lem4}
Let $G$ be a finite group and let $N\unlhd G$. Suppose that $G/N$ is a Frobenius group with kernel $F/N$, $|G:F|=2$ and that $F$ is non-abelian. Then $\pv(G)\geq 13/18.$
\end{lem}

\begin{proof}
Let $H/N$ be the Frobenius complement of $G/N$. Then $|H/N|=2$. Hence $H/N$ is cyclic of order $2$ and so $F/N$ is abelian by \cite[Lemma 7.21]{Isaacs}.  It follows from Lemma \ref{lem3} that $\pv(G/N)=1/2$ and $\van(G/N)=G/N-F/N$ which implies that $G-F\subseteq \van(G)$. 

Write $H=\langle h\rangle N$ for some $h\in H-N.$ Note that $h^2\in N$. Since $F$ is non-abelian, by Corollary \ref{cor1}, there exist normal subgroups $K\unlhd U$ of $F$,  and a character $\psi\in\Irr(F/K)$  with $|F:U|=s>1$ such that $\psi$ vanishes on $F-U$. We may assume that $U$ is maximal normal in $F$ and so $s\ge 2$ is a prime. Hence, either $U\unlhd G$ or $F=UU^h$ and $|L|=|F|/s^2$, where $L:=U\cap U^h\unlhd G.$

If $\psi^h=\psi$, then $\psi$ is $G$-invariant and since $G/F$ is cyclic of order $2$, $\psi$ extends to $\psi_0\in\Irr(G)$ by \cite[Corollary 11.22]{Isaacs}. Thus $F-U\subseteq \van(G)$. Therefore $G-U\subseteq \van(G)$ and $|G:U|=2s.$ Hence $\pv(G)\ge |G-U|/|G|=1-1/(2s)\ge 3/4\geq13/18.$

Assume that $\psi^h\ne \psi.$ Then $I_G(\psi)=F$ and so $\chi=\psi^G\in \Irr(G)$ and $\chi_F=\psi+\psi^h.$ We see that $\psi^h$ vanishes on $F-U^h.$  It follows that  $\chi(x)=\psi(x)+\psi^h(x)=0$ for every $x\in F-(U\cup U^h)$. Thus $F-(U\cup U^h)\subseteq \van(G)$. If $U^h=U,$ then $F-U\subseteq \van(G)$ and we get $\pv(G)\ge 13/18$ as in the previous paragraph. So, assume $U^h\neq U$. Then $F=UU^h$.  Therefore $$|F-(U\cup U^h)|=|F|-|U\cup U^h|=|F|-2|U|+|U\cap U^h|=|G|(s-1)^2/2s^2.$$ Since $|\van(G)|\ge |G-F|+|F-(U\cup U^h)|$,  we have $\pv(G)\ge 1/2+(s-1)^2/2s^2.$

Assume that $s=2$. Then $G/L$ is a $2$-group of order $8$. If $G/L$ is non-abelian, then $\pv(G)\ge \pv(G/L)\ge 3/4>13/18$ by Lemma \ref{lem2}. Assume that $G/L$ is abelian. Then $U\unlhd G$ and thus $U^h=U$, contradicting our assumption $U^h\neq U.$

Assume that $s\ge 3.$ Then $$ \frac{1}{2}+\frac{(s-1)^2}{2s^2}= \frac{1}{2}+\frac{1}{2}\left(1-\frac{1}{s}\right)^2\ge \frac{1}{2}+\frac{1}{2}\left(1-\frac{1}{3}\right)^2=\frac{13}{18}$$ and so $\pv(G)\ge 13/18$ as wanted.
\end{proof}

\begin{proof}[\textbf{Proof of Theorem \ref{th: classification 1/2}}]
Let $G$ be a finite group with $\pv(G)=1/2.$ We will show that $G/\Center(G)$ is a Frobenius group with kernel $F/\Center(G)$, where $|G:F|=2$ and $F$ is abelian. 

Obviously, $G$ is non-abelian. Let $N\unlhd G$ be maximal such that $G/N$ is non-abelian. It follows that $(G/N)'$ is the unique minimal normal subgroup of $G/N$. Write $\overline{G}=G/N.$ By Lemma \ref{lem:quotient}, we have $1/2=\pv(G)\ge \pv(\overline{G})$. By Theorem \ref{th:lower bound}, $\pv(\overline{G})\ge 1/2$ and thus $\pv(\overline{G})=1/2.$ It follows from Proposition \ref{prop:perfect derived subgroup} that  $\overline{G}$ is solvable.

 By Lemma \ref{lem:minimal nonabelian}, either $\overline{G}$ is a non-abelian $p$-group or a Frobenius group whose kernel is an elementary abelian $p$-group, where $p$ is a prime. 
If $\overline{G}$  is a  non-abelian $p$-group, then  $\pv(\overline{G})\ge 3/4$ by Lemma \ref{lem2} and thus $\pv(G)\ge \pv(\overline{G})\ge 3/4>1/2$, a contradiction. Thus  $\overline{G}$ is a Frobenius group. By Lemma \ref{lem3}, $\pv(G)\ge \pv(\overline{G})\ge(n-1)/n$, where $n$ is the order of the Frobenius complement of $\overline{G}$, say $\overline{T}$.  
Since $\pv(G)=1/2$ and $n\ge 2$, we deduce that $n=2.$ Let $\overline{F}$ be the Frobenius kernel of $\overline{G}$. Let $T,F\leq G$ be the full inverse images of $\overline{T},\overline{F}$ in $G$. So $G/N$ is a Frobenius group with kernel $F/N$ and  $|G:F|=2.$  By Lemma \ref{lem4}, $F$ is abelian. 

Write $F=U\times K$, where $U\in\Syl_2(F)$ and $K$ is a $2$-complement of $F$. Since $F\unlhd G$, both $U$ and $K$ are normal in $G$. As $|G:F|=2$, $G/K$ is a $2$-group. By Lemmas \ref{lem:quotient} and \ref{lem2}, $G/K$ is abelian. It follows that $[G,U]\leq K\cap U=1$ and hence $U\leq \Center(G)$. Write $G=F\langle h\rangle$ for some $2$-element $h\in G.$ As $|G:F|=2$, we deduce that $h^2\in U\leq \Center(G)$. Since $K$ is abelian and $\langle h\rangle$ acts coprimely on $K$, we have that $K=[K,\langle h\rangle]\times\Centralizer_K(h)$. Since $[K,h^2]=1$, $h$ inverts each element of $[K,\langle h\rangle]$; hence $\Center(G)=U\times \Centralizer_K(h)$ and $F=\Center(G)\times [K,\langle h\rangle]$; therefore, $G/\Center(G)$ is a Frobenius group with kernel $F/\Center(G)\cong [K,\langle h\rangle],$ where $F$ is abelian and $|G:F|=2$.

For the converse, assume that $G/\Center(G)$ is a Frobenius group with a Frobenius kernel $F/\Center(G)$, where $F$ is abelian and $|G:F|=2$. It follows that $|F/\Center(G)|$ is odd. Hence, if $U\in\Syl_2(F)$, then $U\leq\Center(G)$. Thus $\pv(G)=1/2$  by Lemma    \ref{lem:abelian}.
The proof is now complete.
\end{proof}

\begin{proof}[\textbf{Proof of Theorem \ref{th:small}}]

(i) Let $G$ be a finite group with $\pv(G)\le2/3$. We prove by induction on $|G|$ that $G$ is solvable. Let $N$ be a minimal normal subgroup of $G$. By Lemma \ref{lem:quotient}, we have $2/3\geq \pv(G)\ge \pv(G/N)$ and thus $\pv(G/N)\leq 2/3$. As $|G/N|<|G|$, by induction we deduce that $G/N$ is solvable. 

If $G$ has two distinct minimal normal subgroups, say $N_1$ and $N_2$, then $N_1\cap N_2=1$ and $G/N_i$ is solvable for $i=1,2.$ It follows that $G^{(\infty)}$, the last term of the derived series of $G$, lies in both $N_1$ and $N_2$. Since $N_1\cap N_2=1$, $G^{(\infty)}=1$ and $G$ is solvable. Thus $G$ has a unique minimal normal subgroup, say $M$, and  $G/M$ is solvable. If $M$ is solvable, then $G$ is solvable. So we assume that $M$ is non-solvable. 

If $G/M$ is abelian, then $G'=M$ is perfect and so by Proposition \ref{prop:perfect derived subgroup}, we have $\pv(G)>3/4>2/3$, which is a contradiction. Thus we may assume that $G/M$ is non-abelian. Since $G/M$ is non-abelian solvable, by Lemma \ref{lem:minimal nonabelian}, there exists normal subgroups $M\unlhd N\unlhd G$ of $G$ such that $G/N$ is a non-abelian $p$-group for some prime $p$ or a Frobenius group. If $G/N$ is a $p$-group, then $\pv(G/N)\ge (p^2-1)/p^2\ge 3/4$ by Lemma \ref{lem2}. But this would imply that $\pv(G)\ge \pv(G/N)\ge 3/4>2/3$, a contradiction. Thus $G/N$ is a Frobenius group with a Frobenius complement $H/N$ and kernel $F/N$.  
By Lemma \ref{lem3}, we have  $\pv(G/N)\ge (m-1)/m$, where $m=|H/N|\ge 2.$ If $m\ge 3$, then $\pv(G/N)\ge 2/3$. Moreover, as $N$ is non-solvable, $N\cap\van(G)\neq\emptyset$ by Lemma \ref{lem:defect 0} and so $\pv(G)>2/3$. Thus we can assume that $m=2$. Now by Lemma \ref{lem4}, we have $\pv(G/M)> 2/3$ and thus $\pv(G)>2/3$. This contradiction proves that $G$ is solvable.

\medskip

(ii) Assume that $\pv(G)<2/3.$ We show that $G$ is abelian or $\pv(G)=1/2.$ By part (i), $G$ is solvable. Assume that $G$ is non-abelian. Then $1/2\leq \pv(G)<2/3.$ Let $N\unlhd G$ be maximal such that $G/N$ is non-abelian. As $G/N$ is solvable, $G/N$ is a non-abelian $p$-group for some prime $p$ or $G/N$ is a Frobenius group with kernel $F/N$ and complement $H/N$ by Lemma \ref{lem:minimal nonabelian}. If the former case holds, then $\pv(G)\ge \pv(G/N)\ge 3/4>2/3$ by Lemma \ref{lem2}. Assume now that the latter case holds.
By Lemma \ref{lem3}, we have $\pv(G/N)\ge (m-1)/m$ with $m=|G:F|\ge 2.$ It follows that $m=|G:F|=2$ as $2/3>\pv(G)\ge \pv(G/N).$ By Lemma \ref{lem4}, $F$ is abelian. Finally, by Lemma \ref{lem:abelian}, $\pv(G)=1/2$ and the proof is now complete.
\end{proof}

\subsection*{Acknowledgment} The authors are grateful to  the referee for many helpful suggestions   and for shortening the proofs of Lemma $2.7$ and Theorem $1.3.$




\begin{thebibliography}{100}



\bibitem{DPS} S. Dolfi, E. Pacifici\ and\ L. Sanus, On zeros of characters of finite groups, in {\it Group theory and computation}, 41--58, Indian Stat. Inst. Ser, Springer, Singapore, 2018.



\bibitem{FS} W. Feit\ and\ G. M. Seitz, On finite rational groups and related topics, Illinois J. Math. {\bf 33} (1989), no.~1, 103--131.

\bibitem{GAP} The GAP~Group, \emph{GAP -- Groups, Algorithms, and Programming, Version 4.7.9}; 2015, (\texttt{http://www.gap-system.org}).

\bibitem{GO} A. Granville\ and\ K. Ono, Defect zero $p$-blocks for finite simple groups, Trans. Amer. Math. Soc. {\bf 348} (1996), no.~1, 331--347.

\bibitem{Isaacs} I. M. Isaacs, \emph{Character theory of finite groups.}  AMS Chelsea Publishing, Providence, RI, 2006.

\bibitem{INW} I. M. Isaacs, G. Navarro\ and\ T. R. Wolf, Finite group elements where no irreducible character vanishes, J. Algebra {\bf 222} (1999), no.~2, 413--423.

\bibitem{JK} G. James\ and\ A. Kerber, {\it The representation theory of the symmetric group}, Encyclopedia of Mathematics and its Applications, 16, Addison-Wesley Publishing Co., Reading, MA, 1981. 


\bibitem{Miller} A. R. Miller, Dense proportions of zeros in character values, C. R. Math. Acad. Sci. Paris {\bf 357} (2019), no.~10, 771--772.
 \end{thebibliography}
\end{document}